\documentclass [12pt,english] {article}
\usepackage{amsthm}
\usepackage{amsmath}
\usepackage{amsfonts}
\usepackage{amssymb}
\usepackage[all]{xy}
\usepackage{color}   
\usepackage{tikz} \usetikzlibrary{arrows}
\usepackage{mathrsfs}
\usepackage{listings}
\usepackage{setspace}
\usepackage[font=small,labelfont=bf]{caption}
\singlespace

\newcommand{\C}{\mathscr{C}}

\usepackage[utf8]{inputenc}

\newtheorem{defi}{Definition}
\newtheorem{teo}{Theorem}
\newtheorem{prop}{Proposition}
\newtheorem{lema}{Lemma}
\newtheorem{coro}{Corollary}

\newtheorem{obser}{Remark}
\usepackage[textwidth=2.5cm, textsize=small, colorinlistoftodos]
{todonotes}

\begin{document}
\title{Bikernels by monochromatic paths}

\author{Dennis J. Diaz-Diaz\footnote{Centro de Investigaci\'on en Matem\'aticas, dennis.diaz@cimat.mx}, Isa\'ias F. de-la-Fuente-Jimenez\footnote{Universidad de Guanajuato, isais.delafuente@cimat.mx} \\ Teresa I. Hoekstra-Mendoza\footnote{Centro de Investigaci\'on en Matem\'aticas, maria.idskjen@cimat.mx}, Miguel E. Licona-Vel\'azquez\footnote{Universidad Aut\'onoma Matropolitana Iztapalapa, eliconav23@xanum.uam.mx}, \\ Victoria Terrones-Segura\footnote{Universidad de Guanajuato, victoria.terrones@cimat.mx}}

\date{\empty}

\maketitle
\begin{abstract}
    In this paper, we introduce the concept of bikernel by monochromatic paths of a bicolored digraph. This concept is strongly motivated by the existing notions of kernels, kernels by monochromatic paths, and double stable augmented categories.  We establish sufficient and necessary conditions for several families of bicolored digraphs to have a bikernel by monochromatic paths. Also, we characterize bicolored digraphs without monochromatic cycles that possess a bikernel by monochromatic paths. Similarly, we characterize bicolored digraphs with monochromatic cycles that also have a bikernel by monochromatic paths. Furthermore, we prove sufficient and necessary conditions for some families of digraphs to be $BK$-colorable. This means that a bicoloration of the digraph exists where the resulting bicolored digraph has a bikernel.
\end{abstract}

\section*{Introduction}
For general concepts we refer the reader to \cite{CL}.
Let $D$ be a digraph. A subset $N$ of $V(D)$ is said to be a \textit{kernel} if it is both independent (there are no arcs between any pair of vertices in $N)$ and absorbent (for all $u\in V(D)\setminus N$ there exists $v$ in $N$ such that $(u,v)\in A(D)).$ The concept of kernel was introduced in \cite{VNM} by von Newmann and Morgenstern in the context of Game Theory as a solution for cooperative $n$-players games. Kernels have been studied by several authors, see for example, \cite{HV}, \cite{Matus}, \cite{BG} and \cite{RV}. Chvátal proved in \cite{CV} that recognizing digraphs that have a kernel is an NP-complete problem, so finding sufficient conditions for a digraph to have a kernel or finding large families of digraphs with kernels are two lines of investigation widely studied by many authors. 

A digraph $D$ is said to be \textit{$m$-colored} if the arcs of $D$ are colored with $m$ colors. Let $D$ be an $m$-colored digraph, a path in $D$ is called \textit{monochromatic} if all of its arcs are colored alike. For an arc $(u,v)$ of $D$ we will denote its color by $c_{D}(u,v).$

Let $D$ be an $m$-colored digraph, a subset $K$ of $V(D)$ is said to be a \textit{kernel by monochromatic paths}  if it satisfies the following conditions: (1) there are no monochromatic paths between every pair of vertices in $K$ ($K$ is independent by monochromatic paths) and (2) for every vertex $x\in V(D)\setminus K$ there is an $xK$-monochromatic path ($K$ is absorbent by monochromatic paths). The notion of kernel by monochromatic paths was studied first in \cite{SSW} by Sands, Sauer and Woodrow. The concept of kernel by monochromatic paths generalizes the concept of a kernel since a kernel by monochromatic paths is a kernel when all the monochromatic paths have length one.

Inspired by the concepts of kernel, kernel by monochromatic paths, we will introduce the concept of bikernel by monochromatic paths for bicolored digraphs as follows. 

Given an arc-colored digraph $D$, we say that $D$ is a bicolored digraph if its arcs are colored with colors $1$ and $2$. Denote by $A_i(D)$ the set of arcs which have color $i.$ 
\begin{defi}\label{bikernel}
Let $G$ be a bicolored digraph. We say that a non-empty subset $B\subseteq V(G)$ is a bikernel by monochromatic paths if the following conditions hold
\begin{itemize}
    \item $B$ is independent by monochromatic paths, i.e. for every $u,v\in B$ with $u\neq v$, there does not exist a monochromatic $uv$-path,
    \item $B$ is absorbent with color 1, i.e. for every vertex $v \in V(G)\setminus B$ there exists a monochromatic path from $v$ to a vertex $x \in B$ which has color 1,
    \item $B$ is dominant with color 2, i.e. for every vertex $v \in V(G)\setminus B$ there exists a monochromatic path from a vertex $x \in B$ to $v$ which has color 2.
\end{itemize}
\end{defi}

Figure \ref{1} is an example of a digraph with bikernel where blue arcs have color 1 and red ones have color 2. Notice that the set of the black vertices are the bikernel of this graph. 

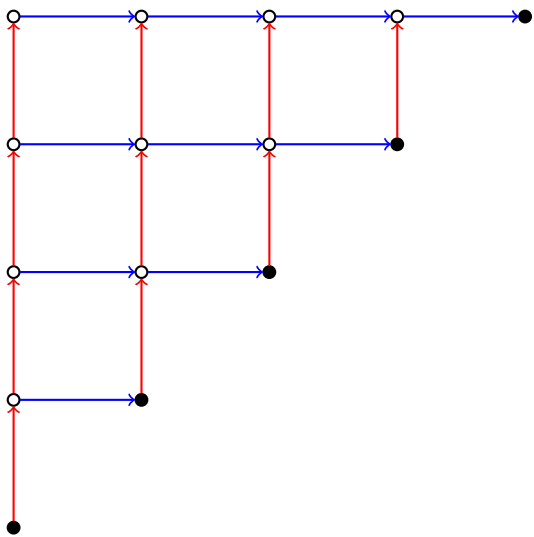
\begin{figure}[h!]
    \centering
\begin{tikzpicture}[scale=1.7,thick]
    \node[draw, circle, scale=.4, fill=black] (12) at (0,0){};
    \node[draw, circle, scale=.4] (13) at (0,1){};
    \node[draw, circle, scale=.4] (14) at (0,2){};
    \node[draw, circle, scale=.4] (15) at (0,3){};
    \node[draw, circle, scale=.4] (16) at (0,4){};
    \node[draw, circle, scale=.4, fill=black] (23) at (1,1){};
    \node[draw, circle, scale=.4] (24) at (1,2){};
    \node[draw, circle, scale=.4] (25) at (1,3){};
    \node[draw, circle, scale=.4] (26) at (1,4){};
    \node[draw, circle, scale=.4, fill=black] (34) at (2,2){};
    \node[draw, circle, scale=.4] (35) at (2,3){};
    \node[draw, circle, scale=.4] (36) at (2,4){};
    \node[draw, circle, scale=.4, fill=black] (45) at (3,3){};
    \node[draw, circle, scale=.4] (46) at (3,4){};
    \node[draw, circle, scale=.4, fill=black] (56) at (4,4){};

    \draw[->, red] (12)--(13);
    \draw[->, red] (13)--(14);
    \draw[->, red] (14)--(15);
    \draw[->, red] (15)--(16);
    \draw[->, red] (23)--(24);
    \draw[->, red] (24)to(25);
    \draw[->, red] (25)to(26);
    \draw[->, red] (34)to(35);
    \draw[->, red] (35)to(36);
    \draw[->, red] (45)to(46);
    \draw[->, blue] (13)--(23);
    \draw[->, blue] (14)--(24);
    \draw[->, blue] (24)--(34);
    \draw[->, blue] (15)--(25);
    \draw[->, blue] (25)--(35);
    \draw[->, blue] (35)--(45);
    \draw[->, blue] (16)--(26);
    \draw[->, blue] (26)--(36);
    \draw[->, blue] (36)--(46);
    \draw[->, blue] (46)--(56);
    
\end{tikzpicture}
    \caption{A digraph with bikernel the black vertices.}
    \label{1}
\end{figure}

\newpage
As seen in Figure \ref{corazon}, the circulatory system can be represented as a digraph. Each organ is a vertex of the digraph and an arc connects two vertices if there is a vein or artery between the corresponding organs. Arteries are colored blue (color 1) and veins red (color 2). Consequently, the heart vertex is a bikernel for this digraph. The concept of a bikernel can model real-world situations, suggesting potential applications in various scientific fields.
\begin{figure}[h!]
    \centering
    \includegraphics[scale=0.5]{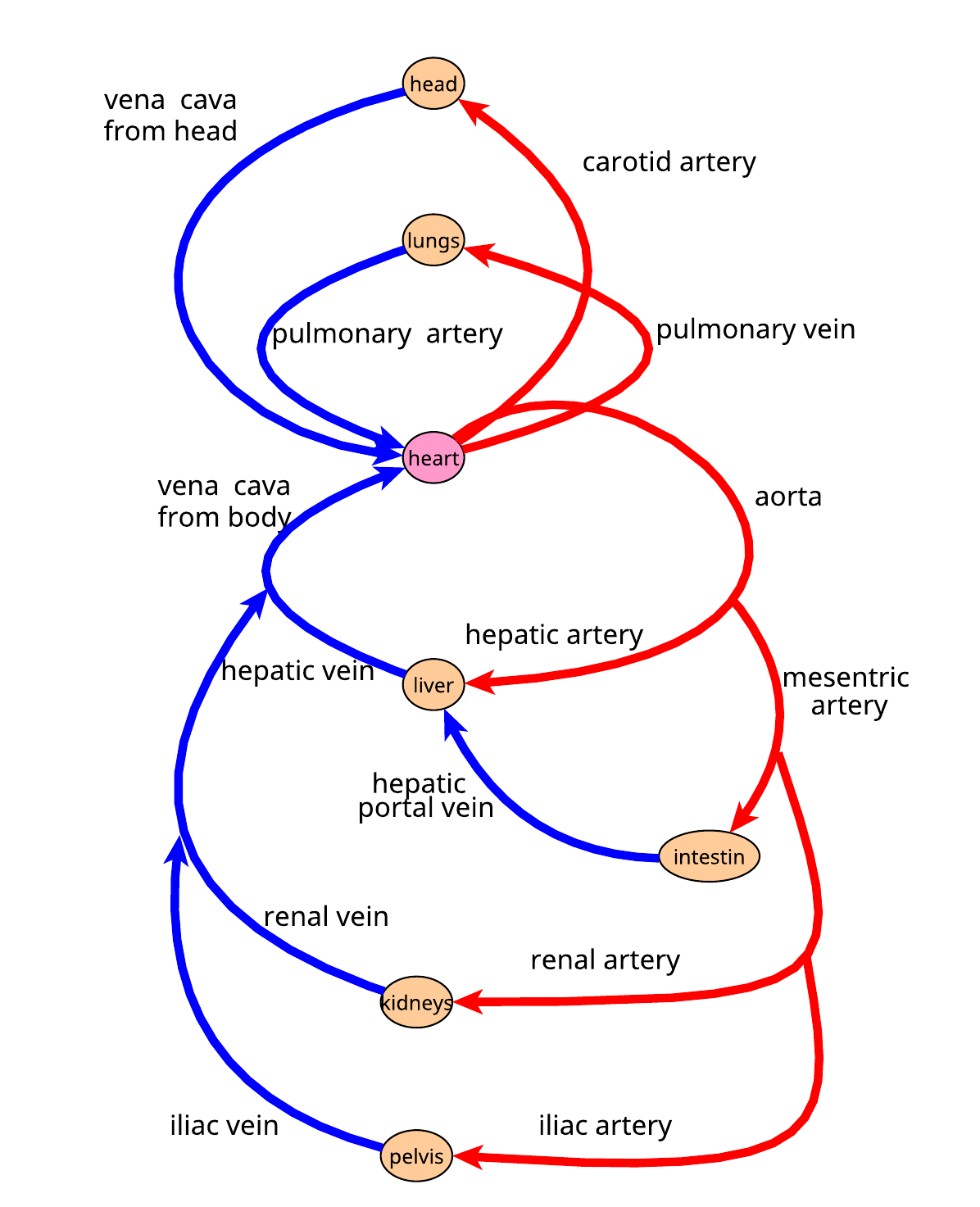}
    \caption{The heart as a bikernel of the blood circulatory system.}
    \label{corazon}
\end{figure}

On the other hand, the concept of bikernel is closely related to that of kernel by monochromatic paths. As seen in Definition \ref{bikernel}, the condition of independence by monochromatic paths is the same for both.  The difference lies in two key properties: a kernel by monochromatic paths is absorbent regardless of the path's color, while a bikernel absorbs by monochromatic paths only at a specific color. Furthermore, a bikernel is dominant by monochromatic paths sets, also at a specific color.

In \cite{SSW}, Sands, Sauer, and Woodrow proved that any $2$-colored (bicolored) digraph has a kernel by monochromatic paths. However, this statement does not hold for bikernels. For example, consider Figure \ref{sinbikernel}, where the white vertex in the digraph $D$ is a kernel by monochromatic paths. Since there is an arc between every pair of vertices, a bikernel for $D$ would also have only one vertex. However, it is easy to see that no vertex exists that absorbs by monochromatic paths in one color and dominates by monochromatic paths the remaining vertices in the other color.

\begin{figure}[h!]
    \centering
    \includegraphics[width=0.4\linewidth]{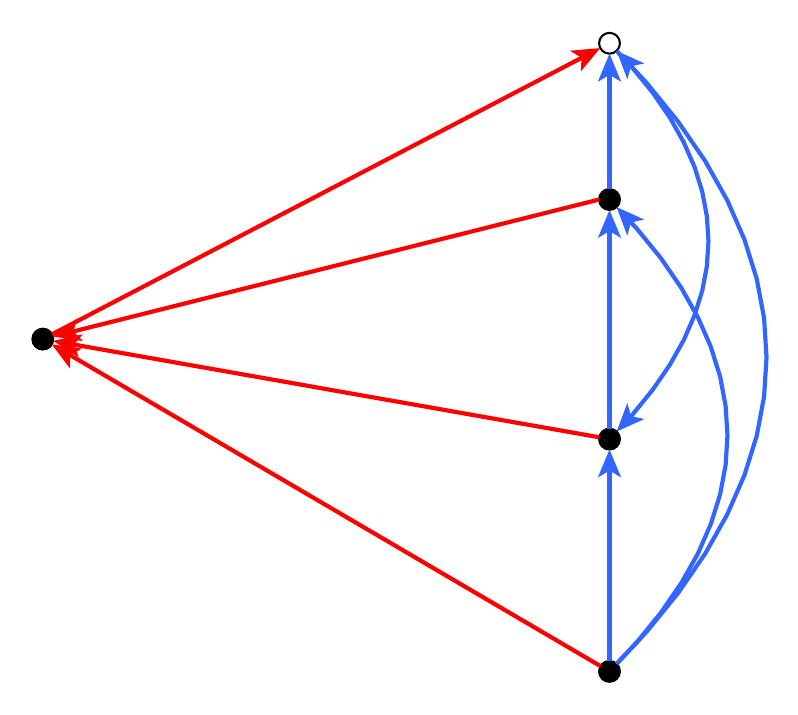}
    \caption{A bicolored digraph $D$ with a kernel by monochromatic paths but without a bikernel.}
    \label{sinbikernel}
\end{figure}


This paper is organized as follows. In Section \ref{Sec1} we work with bicolored paths and cycles, we provide necessary and sufficient conditions on the coloring to ensure the existence of a bikernel by monochromatic paths. Also, in this section, we work with the Cartesian product. A digraph (graph) product \cite{HIK} constitutes a binary operation on digraphs. 
We consider the Cartesian product between paths and between cycles. In both cases, we provide necessary and sufficient conditions to ensure the existence of a bikernel by monochromatic paths. 

Richardson theorem \cite{R} states that every digraph without a directed odd cycle has a kernel. Later, Galeana-Sánchez and Sánchez-López in \cite{HR} proved that given $D$ be a finite $m$-coloured digraph and $\mathcal{C}(D)$ its color-class digraph. If $\mathcal{C}(D)$ has no cycles of odd length at least 3, then $D$ has a kernel by monochromatic paths. They also proved that Richardson’s Theorem is a direct consequence of this result. Following this line of investigation, in Section \ref{caracterizacion} we prove that, given $G$ be a bicolored digraph that does not contain monochromatic directed cycles,  $G$ has a unique bikernel if and only if every critical vertex in $G$ is a supercritical vertex. Also we prove that, if $D$ is a digraph without directed cycles, then $D$ has a $BK$-coloration if and only if $D$ has star-decomposition. On the other hand, 
we also give a characterization of the digraphs having a bikernel in terms of the cyclic classes of the digraph, more specific we show that: let $G$ be a bicolored digraph, $C_1$ and $C_2$ its 1-cyclic classes and 2-cyclic classes and $M_1$, $M_2$ the set of maximal classes in $C_1$ and minimal classes in $C_2$ respectively. Then $G$ has a bikernel if and only there exist a bijection $f:M_1\longrightarrow M_2$ such that $f(m_1)\cap m_1\neq\emptyset$ for all $m_1\in M_1$.  

In Section \ref{Scicloconcuerda}, we investigate bicolored cycles with a chord. We establish necessary and sufficient conditions on both the structure of the digraph and the coloring to guarantee the existence of a bikernel by monochromatic paths.

Category theory was introduced in 1945 by Samuel Eilenberg and Saunders Mac Lane in their foundational work on algebraic topology. A category $\mathscr{C}$ is a pair $(\mathcal{O}b (\mathscr{C}), \mathcal{M}or (\mathscr{C}))$ where $\mathcal{O}b (\mathscr{C})$ is a class called the objects of $\C$ and a class $\mathscr{M}or (\C) $ of elements of the form $f:X \rightarrow Y$ with $X,Y \in \mathcal{O}b(\C)$ called morphism or arrows that have some properties (see Definition \ref{deficat}).  A category is called \textit{small} if the class of objects is a set. If we consider only small categories, one can see a category as a multidigraph which is transitive and has every loop.  On the other hand, a double category is an internal category in $\mathcal{C}at$, the category of all categories. So, a small double category could be seen as a bicolored transitive and reflexive multidigraph (see Definition \ref{doblesmall}). In Section \ref{cat}  and we prove that if $\mathscr{D}$ is an augmented double category, then $\mathscr{D}$ has a bikernel when viewed as a bicolored digraph. It is important to mention that the notion of a bikernel was also highly motivated by the concept of a double stable augmented category, as we shall see in this section. 

\section{Main Results}
The following definitions and result will be useful throughout this paper.
\begin{defi}
Let $G$ be a bicolored digraph. We define the next terms associated to the vertices of the digraph
\begin{itemize}
    \item The $i$-out-degree of a vertex, is given by
    $\delta_{i}^+(v)=|\{ (v,x)\in A(G):(v,x) \in A_i(D)\}|$.
    \item The $i$-in-degree of a vertex, is given by
    $\delta_{i}^-(v)=|\{ (x,v)\in A(G):(x,v) \in A_i(D)\}|$.
    \item The out-degree of a vertex, is given by
    $\delta^+(v)=\delta_1^+(v)+\delta_2^+(v)$.
    \item The in-degree of a vertex, is given by
    $\delta^-(v)=\delta_1^-(v)+\delta_2^-(v)$.
    \item The $i$-degree of a vertex, is given by
    $\delta_{i}(v)=\delta_{i}^+(v)+\delta_{i}^-(v)$.
   
\end{itemize}
\end{defi}

When a vertex has some of the last degrees equal to zero, they receive the following special names.

\begin{defi}
    Let $v$ be a vertex of $G$, if $\delta^+(v)=0$ we say that $v$ is a sink, if $\delta^-(v)=0$ we say that $v$ is a source. Also if $\delta_1^+(v)=0$ we say that $v$ is a 1-sink and if $\delta_2^-(v)=0$ we say that $v$ is a 2-source.
\end{defi}

\begin{lema}\label{d}
    Let $G$ be a digraph whose arcs are colored with two colors: color 1 and color 2. If $G$ has a bikernel $B\subseteq V(G)$, then the following are sufficient conditions for a vertex $v$ to be an element of the bikernel $B$.
    \begin{enumerate}
        \item The vertex $v$ is a 1-sink, this is, $\delta_1^+(v)=0.$
        \item The vertex $v$ is a 2-source, this is, $\delta_2^-(v)=0.$
        \item The vertex $v$ has out-degree, in-degree, 1-degree or 2-degree equal to zero.
    \end{enumerate}
\end{lema}
\begin{proof}
    \begin{enumerate}
        \item Suppose $v\not\in B$, then there must exist a monochromatic path from $v$ to a vertex $b\in B$ of color 1, but this is impossible since $v$ has no outward arcs of color 1. Then $v\in B.$
        \item Suppose $v\not\in B$, then there must exist a monochromatic path from a vertex $b\in B$ to $v$ of color 2, but this is impossible since $v$ has no inward arcs of color 2. Then $v\in B.$
        \item If $v$ has out-degree or 1-degree equal to zero, then it is a 1-sink, and if $v$ has in-degree or 2-degree equal to zero, then it is a 2-source. In both cases, it follows that $v\in B.$
    \end{enumerate}
\end{proof}

If a vertex $v$ of $G$ is a 1-sink or a 2-source we say that $v$ is a critical vertex. From the last lemma we know that every critical vertex must be an element of the bikernel.
\subsection{Paths, cycles and products }\label{Sec1} 
\begin{prop}\label{traybi}
    Let $P_n$ be a directed bicolored path. Then $P_n$ has a bikernel if and only if the following hold:
    \begin{itemize}
        \item[-]  $P_n$ does not contain monochromatic paths of three or more vertices, as induced subdigraphs,
        \item[-] $n$ is odd, and  for $z\in V(P_{n})$ such that $\delta^+(z)=0,$ $c(v,z)=1.$
    \end{itemize}
   
\end{prop}
\begin{proof}
    Assume first that $P_n$ has a bikernel $B$. If $P_n$ contains a monochromatic path of length greater than or equal to two, say  $(x,y,z)$, with color 1. Then $y$ and $z$ are 2-sources, thus, by Lemma \ref{d} $y,z\in B$. Similarly, if $(x,y,z)$ is a path of color 2, then $x$ and $y$ are 1-sinks, thus, by Lemma \ref{d} $x,y\in B$. Therefore, $P_n$ contains at least two adjacent critical vertices, which contradicts that $B$ is an independent set.  Let $v,z\in V(P_{n})$ such that $\delta^+(z)=0$ and $(v,z)\in A(P_{n})$. If $c(v,z)=2,$ then both $v$ and $z$ are 1-sinks, thus, by Lemma \ref{d} $v,z\in B,$ which is again a contradiction.  Let $x,y\in V(P_{n})$ such that $\delta^-(x)=0$ and $(x,y)\in A(P_{n})$, suppose that $n$ is even, since $c(v,z)=1$ and $P_n$ does not contain a monochromatic path of length greater than or equal to two we have that $c(x,y)=1.$ This means that both $x$ and $y$ are 2-sources. Thus, by Lemma \ref{d} $x,y\in B,$ which is a contradiction. Hence, $n$ is odd.

    Now, let $P_{2k+1}$ be a bicolored path holding the hypothesis, with  $V(P_{2k+1})=\{x_1, \dots, x_{2k+1}\}$ and 
    $A(P_{2k+1})=\{(x_i, x_{i-1})\}_{i=2}^{2k+1}$. 
    Then  $B=\{x_1,x_3, \dots, x_{2k+1}\}$ is a bikernel for $P_n.$
\end{proof}
\begin{lema}\label{ciclo}
Let $D$ be a directed cycle whose edges are bicolored. Then $D$ has a unique bikernel if and only if $D$ does not contain monochromatic paths of length two or longer.
\end{lema}

\begin{proof}
    Assume that $D$ has a bikernel $B$ and that $D$ contains two consecutive arcs $(x,y), (y,z)$ having color two. Then $x$ and $y$ must both belong to $B$ since they are both 1-sinks, but this is impossible since they are adjacent. Analogously if $(x,y)$ and $(y,z)$ have color one, then $y$ and $z$ must belong to $B$ since they are both 2-sources, which gives a contradiction.
    Let $V(D)=\{x_1, \dots, x_{2n}\}$ and assume that, for $1\leq i\leq n,$ the arcs $(x_{2i-1}, x_{2i})$ have color one and the remaining arcs have color two. It is clear that $x_{2i}$ is critical. Then $P=\{x_2, x_4, \dots, x_{2n} \} $ is a bikernel for $D$. By Lemma \ref{d} this kernel is unique since any bikernel $B$ must contain the set $P$ and $P$ is a maximal independent set.
\end{proof}

\begin{obser} \label{clasesdecolor}
Let $D$ be a $(2,2)$ regular digraph such that every vertex has an inward and outward arc of each color. Then the following holds:
\begin{enumerate}
    \item The subgraph generated by $A_i(D)$ is the union of disjoint cycles for $ i=1,2.$
    \item If $D$ has a bikernel $B,$ then $B$ must contain a vertex belonging to each of the above mentioned cycles.  
\end{enumerate}
\end{obser}
By Remark \ref{clasesdecolor}, we know that each color class must be a disjoint union of cycles. However, it is possible for these two unions to be of a different number of cycles. The following theorem states that in order to have a bikernel, both color classes must consist of the same amount of disjoint cycles.
\begin{teo}
Let $D$ be a $(2,2)$ regular digraph such that every vertex has an inward and outward arc of each color. If $D$ has a bikernel then $A_i(D)$ is a disjoint union of $k$ cycles for $ i=1,2.$
\end{teo}
\begin{proof}
     Suppose by contradiction that $A_i(D)$ is a disjoint union of $k_i$ cycles for $i=1,2$ with $k_1\leq k_2.$ We may assume without loss of generality that $ k_1<k_2.$ Since $B$ contains one vertex of each of the $k_2$ disjoint cycles of color two, there must exist two vertices which belong to the same monochromatic cycle of color one which is a contradiction.
\end{proof}


Recall the definition of the cartesian product of two digraphs
\begin{defi}
    Let $D_1$ and $D_2$ be two digraphs. We define the cartesian product of this two as the digraph $D_1\square D_2$  given by
    \begin{align*}
        V(D_1\square D_2)=&V(D_1)\times V(D_2)\\
        A(D_1\square D_2)=&\{((u,v_1),(u,v_2)): u\in V(D_1),(v_1,v_2)\in A(D_2)\}\\ 
        &\cup \{((u_1,v),(u_2,v)): v\in V(D_2),(u_1,u_2)\in A(D_1)\}
    \end{align*}
To simplify the notation we write the arcs in $D_1\square D_2$ as $(u_1v,u_2v)$ and $(uv_1,uv_2)$.
\end{defi}
If $D_1$ and $D_2$ are bicolored digraphs then we can induce a 2-coloring
where we color arcs of the form $(u_1v,u_2v)$ with the color of $(u_1,u_2)$ in $D_1$ and arcs of the form $(uv_1,uv_2)$ with the color of $(v_1,v_2)$ in $D_2$.
\begin{teo}\label{prodge}
    Let $G$ and $H$ be two colored digraphs which both have a bikernel. Then the digraph $G\square H$ has a bikernel.
\end{teo}

\begin{proof}
    Let $K_G$ and $K_H$ denote the bikernels of $G$ and $H,$ respectively. Let $K=\{(v,w): v\in K_G, w\in K_H)\}$. We claim that $K$ is a bikernel for $G\square H.$

    Notice that a monochromatic path from $(x,y)$ to $(z,w)$ induces a monochromatic path from $x$ to $z$ if $x\neq z$ or a monochromatic path from $y$ to $w$ if $y\neq w$. Since both $K_G$ and $K_H$ are independent  by monochromatic paths,  $K$ is an independent by monochromatic paths set. Let $(x,y)\in V(G\square H) \setminus K$. There exists a monochromatic $xg$-path of color one for some $g
    \in K_G$ and a monochromatic $yh$-path of color one for some $ h\in K_H.$ This means that there exists a monochromatic path of color one from $(x,y)$ to $(x,h)$ and from $(x,h)$ to $(g,h)$ (we also have a monochromatic path of color one from $(x,y)$ to $(g,y)$ and from $(g,y)$ to $(g,h)$). Analogously there exist a monochromatic $px$-path and $qy$-path of color two for some $p\in K_G$ and $q \in K_H$ which means that there exists a monochromatic $(p,q)(x,y)$-path of color two in $G\square H.$
\end{proof}

In general, the Cartesian product of bicolored digraphs with a bikernel does have a bikernel. For example, the product of a monochromatic cycle of color one with a monochromatic cycle of color two has a bikernel (see Figure \ref{c4}). However, in Theorem \ref{prodtray} we proved that, for bicolored directed paths, the converse of Theorem \ref{prodge}  holds.

\begin{figure}[h!]
    \centering
    \includegraphics[scale=.35]{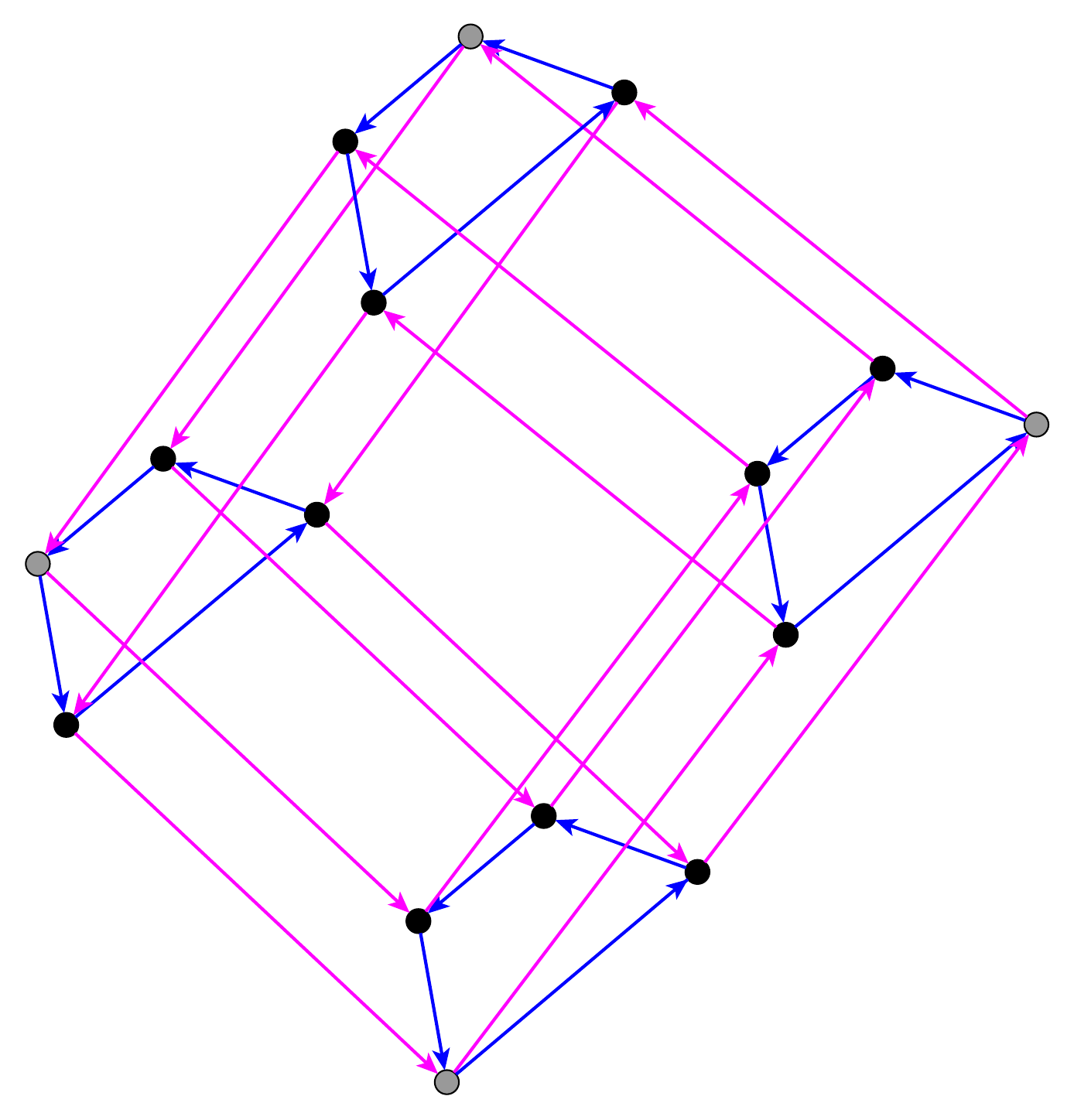}
    \caption{The graph $C_4\square C_4$ where the gray vertices form a bikernel.}
    \label{c4}
\end{figure}

\begin{teo}\label{prodtray}
    Let $G=\gamma_1 \square \dots \square \gamma_n$  where each $\gamma_i$ is a path for every $1\leq i \leq n. $ Then $G$ has a bikernel if and only if every $\gamma_i$ has a bikernel for $ 1 \leq i \leq n.$
\end{teo}

\begin{proof}
Assume that $\gamma_j$ does not have a bikernel for some $j \in \{1, \dots, n\}$ and assume without loss of generality that $j=n.$   Then, $\gamma_n$ contains two adjacent critical vertices $x$ and $y$ such that $\delta^+_1(x)=\delta^+_1(y)=0$ or $\delta^-_2(x)=\delta^-_2(y)=0$. \\
        Let $v_i$ and $w_i$ be the first and last vertex of $\gamma_i$, with $i=1,2,\dots,n-1$. If $\delta^+_1(x)=\delta^+_1(y)=0$ then $(w_1,w_2,\dots, w_{n-1},x)$ and $(w_1,w_2,\dots, w_{n-1},y)$ are two adjacent 1-sinks in $\gamma_1 \square \dots \square \gamma_n$.  Similarly, if $\delta^-_2(x)=\delta^-_2(y)=0$ then $(v_1,v_2,\dots, v_{n-1},x)$ and $(v_1,v_2,\dots, v_{n-1},y)$ are two adjacent 2-sources in $\gamma_1 \square \dots \square \gamma_n$. Hence, $\gamma_1 \square \dots \square \gamma_n$ does not have a bikernel.

        Assume now that every path $\gamma_i$ has a bikernel $B_i$ for $1\leq i \leq n.$ Then, following the proof of Theorem \ref{prodge}, we can see that the set $B=\{(x_1, \dots, x_n): x_i \in B_i\}$ is a bikernel for $G.$
    \end{proof}

\begin{coro}\label{monociclo}
The product of monochromatic cycles $G=C_1 \square C_2$ has a bikernel if and only if $C_1$ and $C_2$ are of the same length and have opposite colors.
\end{coro}

\begin{teo}
Let $C_i$ be a directed cycle for $i=1,2$. Then $G=C_1\square C_2$ has a bikernel if and only if one of the following holds:
\begin{itemize}
    \item Both $C_1$ and $C_2$ have a bikernel, or, 
    \item both $C_1$ and $C_2$ are monochromatic of the same length and have opposite colors.
\end{itemize}
\end{teo}
\begin{proof}
    Assume first that one of the two conditions hold. If both $C_i$ has a bikernel $B_i$ for $i=1,2$ then $B=\{(x,y): x\in B_1, y\in B_2\}$ is a bikernel for $G.$ If the second item holds, by Corollary \ref{monociclo}, $G$ has a bikernel.

    Now, assume that none of the above items hold. Then, either $C_1$ or $C_2$ does not have a bikernel. Assume without loss of generality that $C_1$ does not have a bikernel. Thus, $C_1$ contains a monochromatic path of length at least three $(x,y,z,\dots).$ We can also assume that this path has color one. If $C_2$ contains an arc $(u,v)$ of color one, then $xu$ and $yu$ are two adjacent critical vertices since $\delta_2^+(xu)=\delta_2^+(yu)=0$ thus $G$ does not have a bikernel. 
    This means that $C_2$ is monochromatic of color two. But, by a similar argument, $C_1$ must be monochromatic of color $1.$ Thus by hypothesis, $C_1$ and $C_2$ must have different lengths, hence by Corollary \ref{monociclo}, $G$ does not have a bikernel.
\end{proof}

\subsection{Characterization of digraphs with bikernel}\label{caracterizacion}

The definition of a critical vertex requires for the vertex $v$ to be a 1-sink or a 2-source, this is, to satisfy $\delta_1^+(v)=0$ or $\delta_2^-(v)=0$. Now, if a vertex satisfy both conditions, that is, to be both a 1-sink and a 2-source then we will say that $v$ is a supercritical vertex.

\begin{lema}\label{sinksource}
    Let $G$ be a connected bicolored digraph that does not contain monochromatic directed cycles. If $v$ is not a 1-sink, then there exists a monochromatic $vu$-path of color 1 with $u$ a 1-sink. Similarly, if $v$ is not a 2-source then there exists a monochromatic $uv$-path of color 2 with $u$ a 2-source.
\end{lema}

\begin{proof}
    If $v$ is not a 1-sink then there exists a vertex $v_1$ such that $(v,v_1)$ and $c(v,v_1)=1,$ if $\delta_1^+(v_1)=0$ then we take $u=v_1$.  Otherwise, $\delta_1^+(v_1)>0$ and there exists a vertex $v_2$ such that $(v_1,v_2)$ and $c(v_1,v_2)=1.$ Notice that $v_2$ is different from $v$ and $v_1$ since $G$ does not have monochromatic directed cycles. Continuing this process, we arrive at a vertex $u$ such that $\delta_1^+(u)=0$ and we obtain a $vu$-path of color 1 where $u$ a 1-sink. Analogously, we can prove that, if $v$ is not a 2-source then there exists a $uv$-path of color 2 where $u$ a 2-source. 
\end{proof}

\begin{teo}\label{acyclic}
    Let $G$ be a connected bicolored digraph that does not contain monochromatic directed cycles. Then $G$ has a unique bikernel if and only if every critical vertex in $G$ is a supercritical vertex.
\end{teo}

\begin{proof}
    Suppose $G$ has a bikernel $B\subseteq V(G)$. Let $v\in V(G)$ be a critical vertex, and assume that $v$ is a 1-sink, but $v$ is not a 2-source. By Lemma \ref{sinksource}, there exist $u\neq v$ and $uv$-path of color 2 with $u$ a 2-source.
    Then, by Lemma \ref{d} $u$ and $v$  must belong to $B$ but there is a monochromatic $uv$-path which gives a contradiction. Therefore, $v$ is supercritical. By a similar argument if $v$ is a 2-source, it must also be a 1-sink, then every critical vertex of $G$ is supercritical.
    \medskip\\
    Now suppose that every critical vertex of $G$ is supercritical and consider the set
    $$B=\{v\in V(G): v\text{ is critical}\},$$
    we are going to prove that $B$ is a bikernel of $G$.\\ Take a vertex $v\in V(G)\setminus B$, thus $\delta_1^+(v)>0$ and $ \delta_2^-(v)>0$. Then, there exist $u$ and $w$ with $u$ a 1-sink and $w$ a 2-source such that there exist a $vu$-path of color 1 and a $wv$-path of color 2. So $u,w\in B.$ 
    It remains to prove that $B$ is independent by monochromatic paths. Let $v,v'\in B$. Since $v$ and $v'$ are supercriticals, we have that $\delta_1^+(v)=\delta_1^+(v')=0$ and $\delta_2^-(v)=\delta_2^-(v')=0.$ Thus, there is no monochromatic paths between $v$ and $v'.$  Therefore, $B$ is a bikernel for $G$. 
    
    To prove uniqueness, notice that any bikerkel $B'$ must contain the set of critical vertices $B.$ If $B'\neq B$, then $B'$ contains a non-critical vertex which is both absorbed and dominated by vertices in $B \subseteq B'$ thus $B'$ is not independent. Hence, the bikernel is unique.
    
\end{proof}

\begin{coro}
    Let $T$ be a bicolored directed tree, then $T$ has a bikernel if and only if every critical vertex in $T$ is supercritical.
\end{coro}

\begin{defi}
    We say that a digraph $D$ is $BK$-colorable if there exist a 2-coloring $c: A(D)\rightarrow \{1,2\}$ such that $(D,c)$ has a bikernel.
\end{defi}

\begin{defi}
    We say that a directed tree $T$ is a star if there is some vertex $v\in V(T)$ such that $V(T)=\{v\}\cup N^+(v)\cup N^-(v)$ and $\lvert N^+(v)\rvert, \lvert N^-(v)\rvert>0.$ We name the vertex $v$ the center of the star. 
    Let $S_v$ denote a star centered at the vertex $v.$ A {\em star-decomposition} of a graph $G$ is a set of stars $\{S_{v_1},S_{v_2},\ldots,S_{v_k}\}$ contained in $G$ such that $\left(A(S_{v_1}),A(S_{v_2}),\ldots,A(S_{v_k})\right)$ is a partition of $A(G)$ and $d_G(v_i)=\vert A(S_{v_i}) \vert \geq 2,$ for $1 \leq i \leq k.$
    
\end{defi}
In Figure \ref{starry} we can see an example of a decomposition in stars. We colored the arcs so we can see the stars in it.
\begin{figure}[h!]
    \centering
   \includegraphics[scale=0.5]{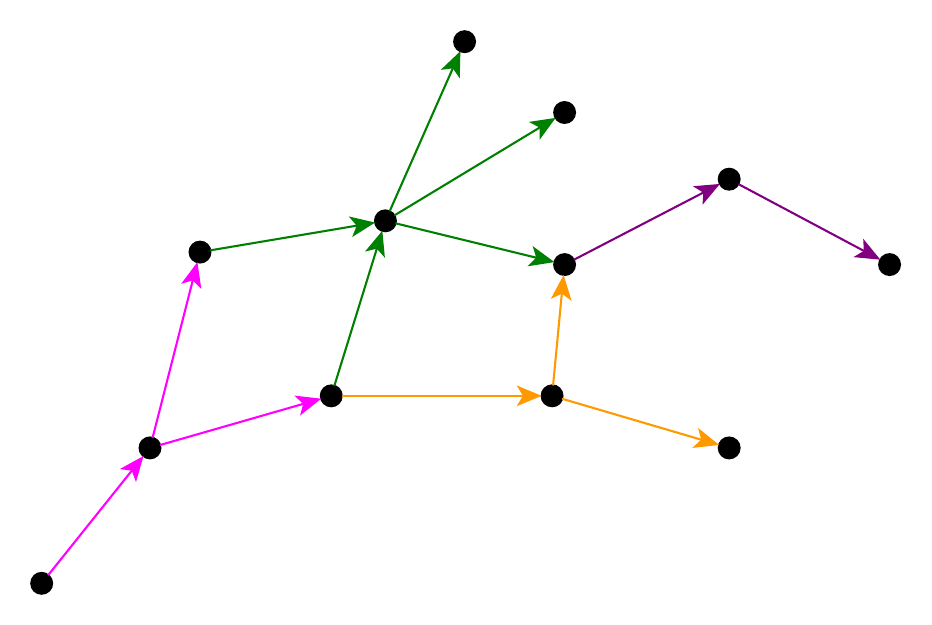}
    \caption{A decomposition in stars}
    \label{starry}
\end{figure}

\begin{teo}
    Let $D$ be a digraph without directed cycles, then $D$ has a $BK$-coloration if and only if $D$ has star-decomposition.
\end{teo}

\begin{proof}
Suppose $D$ has a star-decomposition, $S_{v_{1}},\dots,S_{v_{t}}$ where, for $ 1 \leq i \leq t,$  $v_i$ be the center of $S_i.$  Define $c:A(D)\rightarrow \{1,2\}$ as follows; we color the outward arcs of $v_i$ in $S_i$ with color 1 and the inward arcs of $v_i$ in $S_i$ with color 2. This 2-coloration is well defined because no pair of stars share arcs. Notice that $\delta_1^+(v_i),\delta_2^-(v_i)>0$ for every $ 1\leq i \leq t$. Also, for every $w \in V(D)\setminus\{v_1, \dots, v_t\}$ we have $\delta_1^+(w)=\delta_2^-(w)=0$.
So the condition of Theorem \ref{acyclic} is met and we conclude that $(D,c)$ has a bikernel.\\ 

Now suppose that $(D,c)$ has a bikernel $B\subseteq V(G).$ by Proposition \ref{acyclic} we know that the unique bikernel is $B=\{v\in V(D): v\text{ is critical}\}.$\\
Let $M=V(D)\setminus B$ and notice that $\delta_1^+(v), \delta_2^-(v)>0$, for all $v\in M$.
For each $v\in M$ define $S_v$ as the star such that $$V(S_v)=\{v\}\cup N_1^+(v)\cup N_2^-(v),$$ $$A(S_v)=\{(u,v):u\in N_2^-(v)\}\cup\{(v,u):u\in N_1^+(v)\}.$$  It is easy to see that $D=\bigcup_{v\in M}S_v$ with $S_v$ stars that does not share arcs by pairs.\\ Thus, $D$ has a star-decomposition.
\end{proof}

\begin{coro}
    Let $T$ be a directed tree, then $T$ has a bikernel if and only if $T$ has a decomposition in stars.
\end{coro}

\begin{teo}
    Let $T$ be a tournament with $\lvert V(T)\rvert\geq4$. $T$ is BK-coloreable if and only if $T$ is strongly connected.
\end{teo}

\begin{proof}
   If $T$ is not strongly connected, then there exist two vertices $x, y\in V(T)$ such that $d+(x)=0=d^-(y)$ which means that both $x$ and $y$ are critical. Since, $T$ is a tournament, $T$ any bikernel of $T$ consists of one unique vertex thus cannot contain both $x$ and $y.$
   Now assume that $T$ is strongly connected. Proceed by induction over $n$.
   For $n=4$ we have the following coloring for the strongly connected tournament, where the black vertex is a bikernel.

   \begin{center}
   \begin{tikzpicture}[scale=2]
    \node[circle, draw, fill=black, scale=.5] (1) at (0,0){};
    \node[circle, draw, scale=.5] (2) at (1,0){};
    \node[circle, draw, scale=.5] (3) at (0,1){};
    \node[circle, draw, scale=.5] (4) at (1,1){};
    \draw[->, purple, thick] (1)--(2);
    \draw[->, purple, thick] (2)--(3);
    \draw[->, purple, thick] (3)--(4);
    \draw[->, blue, thick] (3)--(1);
    \draw[->, blue, thick] (4)--(1);
    \draw[->, blue, thick] (2)--(4);
   \end{tikzpicture}
\end{center}

   Now assume that $n>4$ and let $v$ be a vertex such that $T'=T\setminus v$ is strongly connected. Then by induction hypothesis, $T'$ has a coloring $c$ such that $(T',c)$ has a bikernel $k$ consisting of a single vertex. We color every arc $(x,v)$ with color 2 and every arc $(v,y)$ with color one for $x,y\in V(T')$ to obtain a coloring $c'$ of $T$ such that $k$ is a bikernel for $(T,c').$
\end{proof}

    We define the following relations in $V(G),$ for $i\in \{1,2\}$
    \begin{center}
        $u\sim_i v$  if and only if   $u=v$ or $u\neq v$ and there exist an $uv$-path and a $vu$-path of color $i$.\\
        $u\sim v$  if and only if   $u=v$ or $u\neq v$ and there exist an $uv$-path and a $vu$-path.
    \end{center}
    It is easy to see that these are equivalence relations. We consider the equivalence classes, which we call the $i$-cyclic classes and strongly connected components for $i\in \{1,2\}$
        $${C_i}=\{[c]_i:c\in V(G)\},\quad
        {C}=\{[c]:c\in V(G)\}.$$
    Define the following relations, for $i\in \{1,2\},$ in $C_i$ and $C.$
    \begin{center}
        $[u]_i\leq_1 [v]_i$  if and only if $ u=v $ or $u\neq v$ and there exists an $uv$-path of color $i$.\\
        $[u]\leq [v]$  if and only if $ u=v $ or $u\neq v$ and there exist an $uv$-path.
    \end{center} 
    Observe that, each relation defined above, is a partial order.



\begin{teo} 
    Let $G$ be a bicolored digraph, $C_1$ and $C_2$ its 1-cyclic classes and 2-cyclic classes and $M_1$, $M_2$ the set of maximal classes in $C_1$ and minimal classes in $C_2$ respectively. Then $G$ has a bikernel if and only there exist a bijection $f:M_1\longrightarrow M_2$ such that $f(m_1)\cap m_1\neq\emptyset$ for all $m_1\in M_1$.
\end{teo}

\begin{proof}
    Suppose $G$ has a bikernel $B\subset V(G)$. Let $[v]_1\in M_1$ be a maximal element in $C_1$ and  $v'\in [v]_1$. Suppose that there is no element of $B$ in $[v]_1$. Thus, there exists $b\in B$ such that there is a $v'b$-path of color 1. Therefore, $[v]_1\leq_1 [b]_1,$ a contradiction. Hence, every maximal class $[v]_1\in M_1$ contains an element of $B$.      
    Furthermore, since $B$ is an independent by monochromatic paths, every maximal class must contain exactly one element of $B.$ Similarly, we can prove that every minimal class $[v]_2\in M_2$ contains exactly one element of $B$. 
    
    Now, let $B=\{v_1,v_2,\cdots,v_n\}$, we are going to show that every class $[v_i]_1$ is maximal and every class $[v_i]_2$ is minimal. Recall that $\leq_1$ and $\leq_2$ are finite partial orders, so for every class $[v_i]_1$ there must exist a maximal class $[v]_1$ such that $[v_i]_1\leq [v]_1$.
    
     Notice that $[v_i]_1\neq [v]_1$ contradicts the independence by monochromatic paths of $B$ since both $[v_i]_1$ and $[v]_1$ contain an element of $B$. Similarly, for every class $[v_i]_2$ there must exist a minimal class $[v]_2$ such that $[v_i]_2\geq [v]_2$, so $[v_i]_2=[v]_2$. Thus, we have
    $$M_1=\{[v_1]_1,[v_2]_1,\cdots,[v_n]_1\} \text{ and }M_2=\{[v_1]_2,[v_2]_2,\cdots,[v_n]_2\}$$
    and the bijection is clear.\medskip\\
    Now suppose we have a bijection $f:M_1\longrightarrow M_2$ such that  $f(m_1)\cap m_1\neq\emptyset$ for all $m_1\in M_1$. By choosing a vertex in each intersection, we obtain a bikernel for $G.$
\end{proof}

\subsection{Directed cycles with chords}\label{Scicloconcuerda}
Now we are going to see when a directed cycle with a chord has a bikernel. Let $D$ a directed cycle with chords, the base cycle $C$ is the digraph obtained by removing the chords from $D.$  To do this we will use the following lemmas.

\begin{lema}\label{1,2}
    Let $D$ be a bicolored directed cycle with a chord, $B\subset V(D)$ a bikernel of $D$ and $C=(v_1,v_2,\cdots,v_n)$ its the base cycle. If $(x,y,z)$ is a path of length three in $C$ such that $(x,y)$ has color 1 and $(y,z)$ has color 2, then $y\in B$.
\end{lema}

\begin{proof}
    If $y$ is not contained in the chord then $y$ is critical since $y$ has no inward arcs of color 2. If the chord is an outward arc of $y$ then $y$ has no inward arcs of color 2, so $y$ is critical. If the chord is an inward arc of $y$ then $y$ has no outward arcs of color 1, so $y$ is critical. In any case $y$ is critical and by Lemma \ref{d}, $y\in B.$
\end{proof}

\begin{lema}\label{4}
    Let $D$ be a bicolored directed cycle with a chord which has a bikernel. Let $C$ the base cycle of $D.$ Then $C$ has no monochromatic paths of length 4.
\end{lema}
\begin{proof}
    Suppose that there is a monochromatic path of length 4 $(x,y,z,w)$ in $C$. Notice that at most one of these four vertices can be in $B$.
    
     Let us suppose that $(x,y,z,w)$ has color 1. If $(y,w)$ is the chord, then $y$ and $z$ are critical since they have no inward arcs of color 2. Similarly, if $(w,y)$ is the chord, then $z$ and $w$ are critical. In any other case, at most one vertex of $\{y,z,w\}$ is contained in the chord. Thus, at least two vertices in $\{y,z,w\}$  are critical, since they do not have inward arcs of color 2.
     
      Let us suppose that $(x,y,z,w)$ has color 2. Following a similar argument, we can conclude that at least two vertices of the path are critical.  
     Then, in any case we have two vertices in $(x,y,z,w)$ which are critical. Thus, by Lemma \ref{d} they both belong to $B,$ which is a contradiction. Therefore, there are no monochromatic paths of length 4 in $C$.
\end{proof}

\begin{lema}\label{3uni}
    Let $D$ be a bicolored directed cycle with a chord which has a bikernel $B$. Let $C$ the base cycle of $D$. Then $C$ has at most one monochromatic path, $P$ of length 3. Moreover, if $P=(x,y,z)$ exists, the chord in $D$ must contain $y$.   
\end{lema}

\begin{proof}
    Suppose that $C$ has 2 different monochromatic paths of length 3, say $(x_1,y_1,z_1)$ and $(x_2,y_2,z_2)$.
    We work with 2 cases: when both paths have the same color and when one path is of color 1 and the other is of color 2. 
    
    \textbf{Case 1:} Assume that both paths have color 1. Notice that the paths cannot intersect since otherwise there exists a monochromatic path of length 4, which contradicts Lemma \ref{4}. If no vertex of $\{y_i,z_i\}$ is contained in the chord then both $y_i$ and $z_i$ are critical since they do not have inward arcs of color 2 for $ i=1,2$. Thus, the chord must contain a vertex in $\{y_1,z_1\}$ and a vertex in $\{y_2,z_2\}$. 
    Let $(a,b)$ be the chord.  If $b\in \{y_2,z_2\}$, then $(a,b)$ is an outward arc of $y_1$ or $z_1$, so $y_1$ and $z_1$ are critical, since they do not have inward arcs of color 2. Analogously if $b\in \{y_1,z_1\}$ then $y_2$ and $z_2$ are critical, so in both cases we get a contradiction.
    The case where both paths have color 2 is analogous.

    \textbf{Case 2:} 
    Without loss of generality, suppose that $(x_1,y_1,z_1)$ is of color 1 and $(x_2,y_2,z_2)$ is of color 2. 
    Notice that, the outward arc of $z_1$ in the cycle must be of color 2 and the inward arc of $x_2$ in the cycle must be of color 1, otherwise we would have a monochromatic path of length 4. Thus, by Lemma \ref{1,2}, $z_1$ and $x_2$ are elements of the bikernel. Now, the chord must contain $y_1$ and $y_2$ otherwise, the vertex that is not part of the chord is critical, which is a contradiction. So, assume that the chord $c$ connects $y_1$ and $y_2$ in some order. If $c$ is of color 1 then $y_1$ is critical since $y_1$ has no inward arcs of color 2; if $c$ is of color 2 then $y_2$ is critical since $y_2$ has no outward arcs of color 1. Hence, in every case we get a contradiction.\\We conclude that $C$ has at most one monochromatic path of length 3. \\
    Finally, suppose that $C$ has exactly one monochromatic path of length 3, $(x,y,z)$. If $(x,y,z)$ has color one, then the outward arc of $z$ in the cycle is of color two, and it follows that $z$ is in the bikernel. If the chord does not contain $y$, then $y$ is critical and we get a contradiction. Analogously, when $(x,y,z)$ has color two, Therefore, $y$ must be contained in the chord.
\end{proof}

\begin{teo}
Let $D$ be a bicolored directed cycle with a chord with vertices $V(D)=\{x_1, \dots, x_n\}.$ Then $D$ has a bikernel if and only if one of the following conditions holds.
\begin{enumerate}
    \item The base cycle of $D$ does not contain monochromatic paths of length greater than two, thus $n$ is even. Moreover,
    \begin{itemize}
        \item [-] for $1\leq i\leq \frac{n}{2}$ if $(x_{2i-1},x_{2i})$ has color 1 then the chord $(x_i,x_j)$ is such that at least one of the two indexes $i$ an $j$ is odd.
         \item [-] If $i$ is even, then the chord is an outward arc of $x_i$ of color 2. If $j$ is even, then the chord is an inward arc of $x_j$ of color 1.
    \end{itemize}
  
\item There exists a unique monochromatic path of length three $P=(x,y,z)$ contained in the base cycle and 
\begin{itemize}
        \item[-] if $P$ has color one, then the chord is $(v,y)$ for some $v \in V(D)$ and has color two,
        \item[-]  if $P$ has color two, then the chord is $(y,v)$ for some $v \in V(D)$ and has color one.
    \end{itemize}
\end{enumerate}
\end{teo}
\begin{proof}
    First, we are going to prove that if one of the two conditions hold, $D$ has a bikernel. Let $C=(x_1,x_2,\cdots,x_{m})$ the base cycle.
    Suppose that $D$ satisfies condition 1, it is clear that $m=2n$.  We claim that $B=\{x_2,x_4,\cdots,x_{2n}\}$ is a bikernel of $D$. Since $C$ does not contain monochromatic paths of length two or longer, by Lemma \ref{ciclo} $C$ has a bikernel. By hypothesis, for $1\leq i\leq \frac{n}{2},$ $(x_{2i-1},x_{2i})$ has color 1, thus $x_{2i}$ is critical. Therefore,  $B$ is a bikernel for $C.$ We only need to show that $B$ is independent in $D.$ It is easy to see that in any case, the chord does not break the independence of $B$ in $D$ (see Figure \ref{casos1,2,3}). Thus $D$ has a bikernel.
  
  \begin{figure}[h!]
      \centering
\begin{tikzpicture} 
\node (e1) at (0,-0.3){$x_{2k-1}$};
\node (e2) at (0,1.3){$x_{2l+2}$};
\node (e3) at (1.5,-0.3){$x_{2k}$};
\node (e4) at (1.5,1.3){$x_{2l+1}$};
\node (e5) at (3,-0.3){$x_{2k+1}$};
\node (e6) at (3, 1.3){$x_{2l}$};
    \node[circle, draw, scale=.4] (a) at (0,0){};
    \node[circle, draw, scale=.4] (b) at (0,1){};
    \node[circle, draw, scale=.4] (c) at (1.5,0){};
    \node[circle, draw, scale=.4] (d) at (1.5,1){};
    \node[circle, draw, scale=.4] (e) at (3,0){};
    \node[circle, draw, scale=.4] (f) at (3,1){};
    \draw[red,->, thick] (c)--(d);
    \draw[blue,->, thick] (a)--(c);
    \draw[red,->, thick] (c)--(e);
    \draw[blue,->, thick] (d)--(b);
    \draw[red,->, thick] (f)--(d);

\node (e1) at (5,-0.3){$x_{2k-1}$};
\node (e2) at (5,1.3){$x_{2l+2}$};
\node (e3) at (6.5,-0.3){$x_{2k}$};
\node (e4) at (6.5,1.3){$x_{2l+1}$};
\node (e5) at (8,-0.3){$x_{2k+1}$};
\node (e6) at (8, 1.3){$x_{2l}$};
    \node[circle, draw, scale=.4] (a0) at (5,0){};
    \node[circle, draw, scale=.4] (b0) at (5,1){};
    \node[circle, draw, scale=.4] (c0) at (6.5,0){};
    \node[circle, draw, scale=.4] (d0) at (6.5,1){};
    \node[circle, draw, scale=.4] (e0) at (8,0){};
    \node[circle, draw, scale=.4] (f0) at (8,1){};
    \draw[blue,->] (d0)--(c0);
    \draw[blue,->, thick] (a0)--(c0);
    \draw[red,->, thick] (c0)--(e0);
    \draw[blue,->, thick] (d0)--(b0);
    \draw[red,->, thick] (f0)--(d0);

\node (e1) at (10,-0.3){$x_{2k}$};
\node (e2) at (10,1.3){$x_{2l+2}$};
\node (e3) at (11.5,-0.3){$x_{2k+1}$};
\node (e4) at (11.5,1.3){$x_{2l+1}$};
\node (e5) at (13,-0.3){$x_{2k+2}$};
\node (e6) at (13, 1.3){$x_{2l}$};
    \node[circle, draw, scale=.4] (a0x) at (10,0){};
    \node[circle, draw, scale=.4] (b0x) at (10,1){};
    \node[circle, draw, scale=.4] (c0x) at (11.5,0){};
    \node[circle, draw, scale=.4] (d0x) at (11.5,1){};
    \node[circle, draw, scale=.4] (e0x) at (13,0){};
    \node[circle, draw, scale=.4] (f0x) at (13,1){};
    \draw (d0x)--(c0x);
    \draw[blue,->, thick] (a0x)--(c0x);
    \draw[red,->, thick] (c0x)--(e0x);
    \draw[red,->, thick] (d0x)--(b0x);
    \draw[blue,->, thick] (f0x)--(d0x);
\end{tikzpicture} 
\caption{In each figure, we can see the cases for the endpoints of the cord and its coloring.}\label{casos1,2,3} 
 \end{figure}
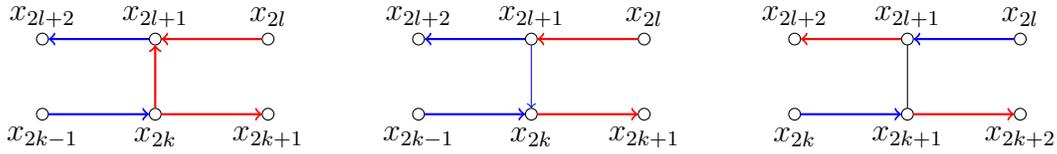

    Now, assume that condition 2 holds. It is clear that, $m=2n+1$. We may assume without loss of generality that $P=(x_1,x_2,x_3)$ and $v=x_i$ with $i>3$.
    Assume first that $(x_1,x_2,x_3)$ is of color 1. We are going to prove that $B=\{x_3,x_5,\cdots,x_{2n+1}\}$ is a bikernel for $D.$ 
Notice that since $B$ is independent in $C$, and by hypothesis the chord $(x_i,x_2)$ has color two, $B$ is also an independent set in $D.$
By Proposition \ref{traybi}, $B$ is a bikernel for $D\setminus \{x_1,x_2\}$. Since the arcs $(x_1,x_2)$, $(x_2,x_3)$ have color one, and $(x_{2n+1},x_1)$ has color two, we only need to show that there exists a monochromatic path of color two from $B$ to $x_2.$ If $i$ is odd, we are done. If $i$ is even, then $(x_{i-1},x_i)$ has color two and $x_{i-1}\in B$ thus $B$ is a bikernel for $D.$ 

Now, assume that $(x_1, x_2, x_3)$ has color two, we are going to prove that $\{x_1, x_4, x_6, \dots, x_{2n}\}$ is a bikernel for $B.$. Since $B$ is independent in $C$ and by hypothesis $(x_2,x_i)$ has color one, $B$ is independent in $D.$ By Proposition \ref{traybi}, $B$ is a bikernel for $D\setminus \{x_2,x_3\}$. Since the arcs $(x_1,x_2), (x_2, x_3)$ have color two and the arc $(x_3,x_4)$ has color one, we only need to show that there exists a monochromatic path from $x_2$ to a vertex of $B.$ If $i$ is even, we are done. If $i$ is odd, the arc $(x_i,x_{i+1})$ has color one with $x_{i+1}\in B.$ Hence, $B$ is a bikernel for $D.$
  
Suppose that $D$ has a bikernel. By lemmas \ref{4} and \ref{3uni}, $C$ does not contain any monochromatic paths of length greater than three, and contains at most one monochromatic path of length exactly three.

Assume first that $C$ contains a unique monochromatic path of length three $(x,y,z)$. Thus $m=|V(D)|$ is even and by Lemma \ref{3uni}, chord has $y$ as one of its endpoints. If $(x,y,z)$ has color one, then $z \in B$ and the chord is $(v,y)$ for some $v\in V(G)$ and has color two, since otherwise $y$ is critical. If $(x,y,z)$ has color two, then $x\in B$ and the chord is $(y,v)$ for some $v\in V(G)$ and has color one, since otherwise $y$ is critical.

    Now suppose that $C$ has no monochromatic paths of length 3, thus $m=2n$ we may assume without loss of generality that $(x_1,x_2)$ is of color 1.  By Lemma \ref{1,2},  $x_{2i}\in B$ for every $ 1\leq i \leq n$ and thus $x_{2i-1}\notin B$ for $2  \leq i \leq n$.  
    
    This means that the chord $(x_i,x_j)$ must be such that at least one of the indexes $i$ and $j$ must be odd, since otherwise we have an arc between two vertices of $B.$
     
    If $i$ is even, then $j$ is odd and $(x_i,x_j)$ must have color 2. Otherwise,  $(x_i,x_j,x_{j+1})$ is a monochromatic path of color 1 between $x_i$ and $x_{j+1}$ with $x_i,x_{j+1}\in B$. Similarly, if $j$ is even, then $i$ is odd and $(x_i,x_j)$ must be of color 1. Hence $D$ satisfies Condition 1 or 2.
\end{proof}

\section{Relationship with Category theory} \label{cat}
The concept of bikernel by monocrokmatic paths was motivated by the concept of augmentation in a double category, as we shall see in this section.
\begin{defi}\label{deficat}
    A category $\mathscr{C}$ is a pair $(\mathcal{O}b (\mathscr{C}), \mathcal{M}or (\mathscr{C}))$ where $\mathcal{O}b (\mathscr{C})$ is a class called the objects of $\C$ and a class $\mathscr{M}or (\C) $ of elements of the form $f:X \rightarrow Y$ with $X,Y \in \mathcal{O}b(\C)$ called morphism or arrows such that:
    \begin{itemize}
        \item for every object $X\in \mathcal{O}b(\C)$ there exists a morphism $1_X:X\rightarrow X$ called the identity morphism on $X,$
        \item given two arrows $f:X \rightarrow Y, g:Y\rightarrow Z$ there exists an arrow $gf: X\rightarrow Z$ and this composition is associative.
        \item given any arrow $f:X \rightarrow Y$, we have $1_yf=f$ and $f1_x=f.$
    \end{itemize}
\end{defi}
For example, we can define the category $BiGra$ whose objects are bicolored digraphs and whose morphisms are color-preserving graph morphisms.

A category is called \textit{small} if the class of objects is a set. If we consider only small categories, one can see a category as a multidigraph which is transitive and has every loop.

Formally, a (small) double category is an internal category in $\mathcal{C}at$, the category of all categories but for our purposes we are going to define a double category as follows.
\begin{defi} \label{doblesmall}
A small double category $\mathscr{D}$ is a tuple $(\mathcal{O}b (\mathscr{D}), \mathcal{H}or (\mathscr{D}), \mathcal{V}er(\mathscr{D}), \mathcal{S}q (\mathscr{D}))$ where $\mathcal{H}or (\mathscr{D})$ and $ \mathcal{V}er(\mathscr{D})$ are two sets of morphisms, $Sq(\mathscr{D})$ is set of diagrams of the form 
\begin{center}
\begin{tikzpicture}
\node (x) at (0,0){$X$};
\node (y) at (0,1){$Y$};
\node (w) at (1,0){$W$};
\node (z) at (1,1){$Z$};
\draw[->, red] (y)--(z);
\draw[->, red] (x)--(w);
\draw[->, blue] (y)--(x);
\draw[->, blue] (z)--(w);
\end{tikzpicture}
\end{center}
such that $(\mathcal{O}b (\mathscr{D}), \mathcal{H}or (\mathscr{D})), (\mathcal{O}b (\mathscr{D}), \mathcal{V}er (\mathscr{D})), (\mathcal{V}er(\mathscr{D})), \mathcal{S}q (\mathscr{D})), (\mathcal{H}or(\mathscr{D}), \mathcal{S}q (\mathscr{D}))$ form categories.
\end{defi}

If we see a small category as a transitive and reflexive multidigraph, a small double category is a bicolored transitive and reflexive multidigraph in which we select a set of subdigraphs which contain exactly two parallel arcs of each color which correspond to the squares and satisfy that:
\begin{itemize}
    \item the subdigraph induced by the arcs of one color $i$ is transitive and reflexive for $i=1,2$, and
    \item the digraph which has a vertices the arcs of a fixed color $i$, and as arcs the squares is also transitive and reflexive for $ i=1,2.$
\end{itemize}

\begin{defi}
    An augmentation of a double category $\mathscr{D}$ consists of a set of objects $A$ such that, for every object $d$ of $\mathscr{D}$ and $a, a' \in A$, there exists a unique morphism from  $a$  to $d$ in $\mathcal{H}or\mathscr{D}$ and a unique morphism from $d $ to $a'$ in $\mathcal{V}er\mathscr{D}$.
We say that a category is pointed if it contains an augmentation consisting of a single vertex.
\end{defi}
We can translate the definition of an augmentation to graph theory. Let $A$ be a not necessarily independent set of objects (vertices), such that its arcs are colored with two colors. \textit{An augmentation} is a subset of $A$ which is uniquely absorbent in one color and uniquely dominant in the other color. In particular, instead of a category, let us consider a two-colored digraph which is not necessarily transitive. In this case, we can consider absorbance by monochromatic paths of color two and dominance by monochromatic paths of color one. Moreover, if the digraph is not transitive, we can require independence instead of uniqueness. So we obtain the definition of a bikernel in a digraph. Thus, we can prove following theorem.

\begin{teo}
    Let $\mathscr{D}$ be an augmented double category. Then $\mathscr{D}$ has a bikernel when viewed as a bicolored digraph.
\end{teo}
\begin{proof}
    We are going to prove that the augmentation $\mathcal{A}$ is a bikernel. Since $\mathcal{A}$ is both dominant in color two and absorbent in color one, we only need to show that $\mathcal{A}$ is independent by monochromatic paths. Thus assume by contradiction that there exists a monochromatic path of color $i$ between two elements $x,y$ in $\mathcal{A}$. Since the subgraph induced by the arcs of color $i$ is transitive, we have that the arc $(x,y)$ exists. Since $\mathscr{D}$ is reflexive in both colors, we have that both $x$ and $y$ are both dominated and absorbed by themselves, thus the arc $(x,y)$ contradicts the uniqueness condition of the augmentation $\mathcal{A}.$ Hence $\mathcal{A}$ is a bikernel.
\end{proof}


\begin{thebibliography}{99}
\bibitem{AL}
Arpin P. and Linek  V.:
\newblock {Reachability problems in edge-colored digraphs},
\newblock {Discrete Math. 307}, 2276-2289 (2007)

\bibitem{BG}
Boros E., Gurvich V.:
\newblock {Perfect graphs, Kernels and cores of cooperatives games},
\newblock {RURCOR Research Report 12},  
\newblock {Rutgers University, (2003)}

\bibitem{CL}
Chartrand G., Lesniak L.:
\newblock {Graphs and digraphs},
\newblock {4th edition},  
\newblock {Chapman and Hall/CRC},
\newblock {Boca Raton, (2005)}

\bibitem{CV}
Chvátal V.:
\newblock {On the computational complexity of finding a kernel},
\newblock {Report CRM300},  
\newblock {Centre de Reserches Matématiques},
\newblock {Université de Montréal, (1973)}

\bibitem{H}
Galeana-Sánchez H.: 
\newblock {Kernels in edge-colored digraphs},
\newblock {Discrete Math. 184, 87–99 (1998).}

\bibitem{HV}
Galeana-Sánchez H.,  Neumann-Lara  V.: 
\newblock {On kernels and semikernels of digraphs},
\newblock {Discrete Math. 48,} 67–76(1984)

\bibitem{HR}
Galeana-Sánchez H., Rojas-Monroy R.:
\newblock{Kernels and some operations in edge-colored digraphs}, 
\newblock{Discrete Math. 308,} 6036-6046 (2008)


\bibitem{HRR}
Galeana-Sánchez H., Rojas-Monroy R., Sánchez-López R., Zavala-Santana B.:
\newblock{$H$-kernels by walks in the subdivision digraph, Transactions on Combinatorics, Vol. 9, No. 2, 61-75 (2020)} 

\bibitem{HR}
Galeana-Sánchez, H., Sánchez-López, R. An Extension of Richardson’s Theorem in m-Colored Digraphs. Graphs and Combinatorics 31, 1029–1041 (2015). https://doi.org/10.1007/s00373-014-1412-6

\bibitem{HRS}
Galeana-Sánchez H., Sánchez-López R.:
\newblock{Kernels by monochromatic paths and color-perfect digraphs,
Discussiones Mathematicae Graph Theory 36, 309–321(2016)}

\bibitem{HIK}
Hammack, R., Imrich, W., Klavzar, S., Handbook of Product Graphs. CRC Press, Florida (2011).

\bibitem{Matus}
Harminc M.:
\newblock{Solutions and Kernels of a Directed Graph}, 
\newblock{Math Slovaca}, 
\newblock{Vol. 32, No. 3, 263-267 (1982)}


\bibitem{VNM}
von Neumann J.,  Morgenstern O.:
\newblock{Theory of games and economics behaviour},
\newblock{Princeton University Press},
\newblock{Princeton, (1994)}

\bibitem{R}
 M. Richardson.: On weakly ordered systems, Bull. Amer. Math. Soc. 52, 113-116 (1946).


\bibitem{RV}
Rojas-Monroy R.,  Villarreal-Valdés J. I.:
\newblock{Kernels in infinite digraphs}, 
\newblock{AKCE J. Graphs. Combin, 7}, 103-111 (2010)


\bibitem{SSW}
Sands B., Sauer N., Woodrow R.:
\newblock{On monochromatic paths in edge coloured digraphs}, 
\newblock{J. Combin . Theory Ser. B 33}, 271-275 (1982)


\bibitem{J}
Topp J.:
\newblock{Kernels of digraphs formed by some unary operations from other digraphs,} 
\newblock{J. Rostock Math}.
\newblock{Kolloq. 21}, 73-81 (1982)


\end{thebibliography}
\end{document}